\documentclass{amsart}

\usepackage{amssymb,amsthm}
\usepackage{braket}
\usepackage{enumitem}

\DeclareFontFamily{U}{mathx}{\hyphenchar\font45}
\DeclareFontShape{U}{mathx}{m}{n}{<-> mathx10}{}
\DeclareSymbolFont{mathx}{U}{mathx}{m}{n}
\DeclareMathAccent{\widebar}{0}{mathx}{"73}

\newtheorem{theorem}{Theorem}[section]
\newtheorem{remark}[theorem]{Remark}
\newtheorem{lemma}[theorem]{Lemma}
\newtheorem{corollary}[theorem]{Corollary}
\newtheorem{proposition}[theorem]{Proposition}
\newtheorem{question}{Question}

\theoremstyle{definition}
\newtheorem{definition}[theorem]{Definition}
\newtheorem{notation}[theorem]{Notation}
\newtheorem{example}[theorem]{Example}

\newenvironment{sproof}{\paragraph{Proof of Claim}}{\hfill \qedsymbol \text{ {\footnotesize  Claim}}}

\newcommand{\twopartdef}[4]
{
	\left\{
	\begin{array}{ll}
		#1 & \mbox{if } #2 \\
		#3 & \mbox{if } #4
	\end{array}
	\right.
}
\newcommand{\threepartdef}[6]
{
	\left\{
	\begin{array}{ll}
		#1 & \mbox{if } #2 \\
		#3 & \mbox{if } #4 \\
		#5 & \mbox{if } #6
	\end{array}
	\right.
}

\newcommand{\cM}{\mathcal{M}}
\newcommand{\cN}{\mathcal{N}}
\newcommand{\cL}{\mathcal{L}}

\newcommand{\cA}{\mathcal{A}}
\newcommand{\cB}{\mathcal{B}}
\newcommand{\cC}{\mathcal{C}}
\newcommand{\cD}{\mathcal{D}}
\newcommand{\cG}{\mathcal{G}}

\newcommand{\vphi}{\varphi}

\newcommand{\scomp}[2]{#1[#2]}
\newcommand{\gcompMN}{\scomp{\cM}{\cN_a}_{a\in\cM}}
\newcommand{\closure}[1]{\widehat{#1}}

\newcommand{\into}{\hookrightarrow}

\DeclareMathOperator{\aut}{Aut}
\DeclareMathOperator{\Th}{Th}
\DeclareMathOperator{\age}{age}

\title{On products of elementarily indivisible structures}
\author{Nadav Meir}

\address{Department of Mathematics\\ Ben-Gurion University of the Negev\\
	P.O.B 653, Be'er Sheva 8410501, ISRAEL}

\email{mein@math.bgu.ac.il}


\begin{document}

	\begin{abstract}
		We say a structure $\mathcal{M}$ in a first-order language $\mathcal{L}$ is \emph{indivisible} if for every coloring of its universe in two colors, there is a monochromatic substructure $\mathcal{M}^{\prime} \subseteq \mathcal{M}$ such that $\mathcal{M}^{\prime} \cong \mathcal{M}$. Additionally, we say that $\mathcal{M}$ is \emph{symmetrically indivisible} if $\mathcal{M}^{\prime}$ can be chosen to be \emph{symmetrically embedded} in $\mathcal{M}$ (that is, every automorphism of $\mathcal{M}^{\prime}$ can be extended to an automorphism of $\mathcal{M}$). Similarly, we say that $\mathcal{M}$ is \emph{elementarily indivisible} if $\mathcal{M}^{\prime}$ can be chosen to be an elementary substructure. We define new products of structures in a relational language. We use these products to give  recipes for construction of elementarily indivisible structures which are not transitive and elementarily indivisible structures which are not symmetrically indivisible, answering two questions presented by A. Hasson, M. Kojman and A. Onshuus.
		
	\end{abstract}

	\maketitle 

	\section{Introduction}
	The notion of indivisibility of relational first-order structures and metric spaces is well studied
	in Ramsey theory. (\cite{Metrics07}, \cite{ElSa93}, \cite{ElSa91}, and \cite{KoRo86}  are just a few examples of the extensive study in this area.) Recall that a
	structure $\cM$ in a relational first-order language is indivisible, if for every coloring of its universe in two colors, there is a monochromatic substructure $\cM^{\prime} \subseteq \cM$ such that $\cM^{\prime} \cong \cM$. Rado's random graph, the ordered set of
	natural numbers and the ordered set of rational numbers are just a few of the many
	examples. Weakenings of this notions have also been studied (see \cite{Sa14}). A known extensively studied strengthening of this notion is the pigeonhole property (see \cite{Pig1}, \cite{Pig2}).
	For an extensive survey on indivisibility see \cite[Appendix A]{Fr00}. 
	
	In \cite{GeKo11}, several induced Ramsey theorems for graphs were strengthened to a ``symmetrized'' version, in which the induced monochromatic subgraph
	satisfies that all members of a prescribed set of its partial isomorphisms extend to automorphisms
	of the original graph. In \cite{HKO11}, following \cite{GeKo11}, a new strengthening of the notion of indivisibility was introduced: 
	\begin{definition}
		We say a substructure $\cN \subseteq \cM$ is \emph{symmetrically embedded} in $\cM$ if every automorphism of $\cN$ extends to an automorphism of $\cM$.
		
		We say that $\cM$ is \emph{symmetrically indivisible} if for every coloring of its universe in two colors, there is a monochromatic $\cM^{\prime} \subseteq \cM$ such that $\cM^{\prime}$ is isomorphic to $\cM$ and $\cM^{\prime}$ is symmetrically embedded in $\cM$. 
	\end{definition}
	
	In \cite{HKO11}, several examples of symmetrically indivisible structures were investigated. Examples include the random graph (\cite{GeKo11}), the ordered set of rational numbers, the ordered set natural numbers, and the universal $n$-hypergraph.
	
	In the last section of \cite{HKO11}, another strengthening of the notion of indivisibility was introduced:
	
	\begin{definition}
			we say that $\cM$ is \emph{elementarily indivisible} if for every coloring of its universe in two colors, there is a monochromatic $\cM^{\prime} \subseteq \cM$ such that $\cM^{\prime}$ is isomorphic to $\cM$ and $\cM^{\prime}$ is an elementary substructure of $\cM$.
	\end{definition}
 Classic examples for this notion, as given in \cite{HKO11}, are the random graph, the ordered set of rational numbers, and the \emph{colorful graph} described below.
 
 \begin{example}\label{ex:Colorful}
 		Let $\cC$ be the class of all finite complete graphs with edges colored in $\omega$-many colors, i.e. all finite $\{R_i\}_{i\in \omega}$-structures such that $\{R_i\}_{i\in \omega}$ are disjoint, irreflexive, symmetric binary relations whose union is the complete graph. This is a Fra\"iss\'e class and we call its Fra\"iss\'e limit \emph{the colorful graph}.
 	\end{example}
 
  In Example \ref{ex:EInotQE} we present an example from \cite{HKO11} of an elementarily indivisible structure which does not admit quantifier elimination and is, in fact, not even model complete. Despite this example, in view of Lemma \ref{TMI}, every elementarily indivisible structure can be considered as a reduct of an indivisible structure who admits quantifier elimination.
   
 A classic example of a symmetrically indivisible structure which is not elementarily indivisible is the ordered set of natural numbers, since every singleton is $\emptyset$-definable.
	(In fact, there are no proper elementary substructures of $\langle \omega, <\rangle$.)
	
	In view of the above example, indivisibility should be viewed as a property of  the pair $(\cM, \cL)$ of a structure and the language in which it is given. Elementary indivisibility seems to be the right analogous property of the structure only (i.e., independent of its language). This statement is given a precise meaning in Lemma \ref{TMI}.

	In \cite{HKO11}, The following questions were asked regarding the properties of elementarily indivisible structures, as well as the relation between this notion and the notion of symmetric indivisibility:
	
	\begin{question}\label{q:elemSym}  Does elementary indivisibility imply symmetric indivisibility?
	\end{question}
	\begin{question}\label{q:elemHom} Is every elementarily indivisible structure homogeneous?
	\end{question}
	
	\begin{question}\label{q:elemRigid} Is there a rigid elementarily indivisible structure?
	\end{question}

	In the literature the precise definition of homogeneity tends to vary; for example, in \cite{Mac11}, a structure is said to be homogeneous if it is what we call ultrahomogeneous. Here we follow the conventions of \cite{Hodges93}  and \cite{Marker02}, as presented in Definitions \ref{defHom} and \ref{defUltrahom}.
	
	To quote  \cite{Metrics07} in a similar context, ``The uncountable case is different as the indivisibility property may fail badly". In view of this, since the dawn of mankind (i.e. all the study mentioned above), indivisibility of first-order structures has been mostly studied in the countable context, since in the uncountable case set theoretic phenomena come into play. We note that while all results mentioned in this paper hold under the restriction to countable structures, in fact the countability assumption is superfluous.

	In this paper, we investigate a construction we call the lexicographic product $\cM[\cN]$ of two relational structures $\cM$ and $\cN$, presented in Definition \ref{defDefProduct}. We note this construction is very similar to the ``composition'' defined in \cite{HKO11} and it generalizes the lexicographic order and the lexicographic product of graphs, as known in graph theory. In Section \ref{sectionQE}, We show that if $\cM$ and $\cN$ both admit quantifier elimination and every two singletons in $\cM$ satisfy the same first-order formulas (i.e. the theory of $\cM$ is \emph{transitive} in the sense of Definition \ref{def:transitiveTheory} below), then $\scomp{\cM}{\cN}^s$ admits quantifier elimination as well. We use this result to show that if $\cM$ and $\cN$ are both elementarily indivisible, then so are $\scomp{\cM}{\cN}$ and $\scomp{\cM}{\cN}^s$.

	We further generalize the quantifier elimination result to a generalized product construction we introduce in Definition \ref{defGenProduct}.

	Applying the results mentioned above, in Section \ref{sectionElemNonTrans} we give general constructions of elementarily indivisible structures which are not transitive and in Section \ref{elemNonSym}  of elementarily indivisible structures which are not symmetrically indivisible, answering Questions   \ref{q:elemSym} and \ref{q:elemHom} negatively. Question \ref{q:elemRigid} remains open.

\subsection{Preliminaries}

Unless otherwise specified, we do not distinguish between a structure $\cM$ and its universe (or underlying set). Throughout this paper all languages are relational so there is no distinction between subsets and substructures of a given structure. The notation for both is $B\subseteq \cM$. We denote the cardinality of a structure $\cM$ by $|\cM|$.

\begin{definition}
	If $\cM$ and $\cN$ are $\cL$-structures and $B \subseteq \cM$, we say that $f : B \to \cN$ is a
	\emph{partial elementary} map if
	$\cM \models \vphi\left(\bar{b}\right) \iff \cN\models \vphi\left(f\left(\bar{b}\right)\right)$
	for all $\cL$-formulas $\vphi$ and all finite sequences $\bar{b}$ from B.

	If $B=\cM$ we just say $f$ is an \emph{elementary embedding}.
	
	A substructure $\cM\subseteq \cN$ is an \emph{elementary substructure} if the inclusion map $\iota$ is an elementary embedding, in which case we denote $\cM\preceq \cN$.
\end{definition}

\begin{definition}
	If $\cM$ and $\cN$ are $\cL$-structures and $B \subseteq \cM$, we say that $f : B \to \cN$ is a
	\emph{partial isomorphism} if
	$\cM \models \vphi\left(\bar{b}\right) \iff \cN\models \vphi\left(f\left(\bar{b}\right)\right)$
	for all \emph{quantifier-free} (or equivalently, atomic) $\cL$-formulas $\vphi$ and all finite sequences $\bar{b}$ from B.
	
\end{definition}

\begin{definition}\label{defHom}
	We say a structure $\cM$ is \emph{homogeneous} if whenever $A \subset \cM$ with $|A| < |\cM|$ and $f : A \to \cM$ is a partial
	elementary map, there is an automorphism $\sigma\in \aut(\cM)$ such that $\sigma \upharpoonright A = f$.
\end{definition}

\begin{definition}\label{defUltrahom}
	We say a structure $\cM$ is \emph{ultrahomogeneous} if whenever $A \subset \cM$ with $|A| < |\cM|$ and $f : A \to \cM$ is a partial
	isomorphism, there is an automorphism $\sigma\in \aut(\cM)$ such that $\sigma \upharpoonright A = f$.
\end{definition}

In \cite{HKO11}, a construction very similar to the following was introduced. We note that while our construction is slightly different, in fact, in the context of binary relational languages these two definitions coincide. We further note that the notation below is classic for structures in a binary language (e.g. \cite{Cherlin} and \cite{Lachlan}) and in fact coincides with the lexicographic (partial) order and the lexicographic product of graphs (sometimes referred to as wreath product). Having said that, \cite[Definition 2.1]{HKO11} and  \cite[Definition 2.3]{HKO11} are the earliest occurrences the author was able to find of such a product for languages of arbitrary arity.

\begin{definition}\label{defDefProduct}
	Let $\cM$, $\cN$ be structures in a relational language, $\cL$. Let $M, N$ be their universes, respectively.
	The \emph{lexicographic product} $\cM[ \cN ]$ is the $\cL$-structure whose universe is $M\times N$ where for every
	$n$-ary relation $R\in \cL$ we set
	\begin{align*}
	& R^{\cM [ \cN ]} := \\
	&\Set{\big(\left(a_1,b_1\right),\dots,\left(a_n,b_n\right)\big)\ |\ \begin{matrix} \bigwedge_{1\leq i,j\leq n} a_i=a_j \text{\ \ \  and \ \ \  }  \cN \models R\left(b_1,\dots, b_n\right) \end{matrix} } \ \  \cup \\ 
	& \Set{\big(\left(a_1,b_1\right),\dots,\left(a_n,b_n\right)\big)\ |\ \begin{matrix} \bigvee_{1\leq i\neq j\leq n} a_i\neq a_j \text{\ \ \ \  and \ \ \ \ } \cM \models R\left(a_1,\dots, a_n\right) \end{matrix} }.
	\end{align*}

	Let $\scomp{\cM}{\cN}^{s}$ be $\scomp{\cM}{\cN}$
	expanded by a binary relation $s\notin \cL$ interpreted as 
	$$\Set{\big(\left(a_1,b_1\right),\left(a_2,b_2\right)\big)\in (M\times N)^2 | a_1 = a_2}.$$
	
\end{definition}

For the purposes of this paper, we generalize the definition above to the following.

\begin{definition}\label{defGenProduct}
	Let $\cM$, $\{\cN_a\}_{a\in \cM}$ be structures in a relational language, $\cL$. Let $M, \{\cN_a\}_{a\in \cM}$ be their universes, respectively.
	The \emph{generalized product} $\gcompMN$ is the $\cL$-structure whose universe is $\bigcup_{a\in \cM} \{a\}\times N_a$ where for every
	$n$-ary relation $R\in \cL$ we set
	\begin{align*}
	& R^{\gcompMN} := \\  
	& \Set{\big((a,b_1),\dots,(a,b_n)\big)\ |\  a\in M \text{\ \ \ and\ \ \ } \cN_a \models R(b_1,\dots, b_n)  } \ \  \cup \\ 
	& \Set{\big((a_1,b_1),\dots,(a_n,b_n)\big)\ |\ \begin{matrix} \bigvee_{1\leq i\neq j\leq n} a_i\neq a_j \text{\ \ \ \  and \ \ \ \ } \cM \models R(a_1,\dots, a_n) \end{matrix} }.
	\end{align*}

	Let $\gcompMN^s$ be $\gcompMN$
	expanded by a binary relation $s\notin \cL$ interpreted as 
	$$\Set{\big((a,b_1),(a,b_2)\big) | a\in \cM \text{\ \ and\ \ } b_1,b_2\in \cN_a}.$$
\end{definition}

Note that if there is a fixed $\cN$ such that  $\cN_a  = \cN$ for all $a\in \cM$, then this definition coincides with $\scomp{\cM}{\cN}$ and $\scomp{\cM}{\cN}^s$.

\begin{remark}\label{unary}
	Notice that the interpretation of unary predicates in the product does not depend on their interpretation in $\cM$, i.e. for a unary predicate $U\in \cL$, \[\gcompMN \models U((a,b)) \iff \cN_a \models U(b). \]
\end{remark}

\begin{remark}
	Notice that if  $\cM$, $\{\cN_a\}_{a\in M}$ are structures in a relational language $\cL$ and $a\in \cM$, then the substructure  $\{a\}\times \cN_a$ is isomorphic to $\cN_a$.
\end{remark}

\begin{definition}\label{def:transitiveTheory}
	We say a theory $T$ is \emph{transitive} if for every $\phi(x)$ in one free variable, either $T \models \forall x\, \phi(x)$
	or $T \models \forall x\, \neg \phi(x)$ (i.e. $T$ has a unique $1$-type).
\end{definition}

\begin{definition}\label{def:transitiveStructure}
	We say an $\cL$-structure $\cM$ is \emph{transitive} if for every $x, y\in \cM$, there is an automorphism $\sigma\in \aut(\cM)$ such that $\sigma(x)=y$.

\end{definition}
\begin{lemma}\label{elemTrans}
	$\Th(\cM)$ is transitive for every elementarily indivisible $\cL$-structure $\cM$.
\end{lemma}

\begin{proof}
	If $\Th(\cM)$ is not transitive, then there is an $\cL$-formula in one free variable $\phi(x)$ such that $\Th(\cM)\not \models \forall x\, \phi(x)$ and $\Th(\cM)\not \models \forall x\, \neg\phi(x)$.
	By completeness, $\Th(\cM)\models \exists x\, \neg\phi(x)$ and $\Th(\cM)\models \exists x\, \phi(x)$. Define a coloring $c:M\to\{\text{red},\text{blue}\}$ as follows:
	\[ c(x):=\twopartdef{\text{blue}}{\cM\models \phi(x)}{\text{red}}{\cM\models \neg\phi(x).} \]
	It is clear that no $c$-monochromatic substructure is elementary.
\end{proof}

Note that obviously if $\cM$ is a transitive structure, then $\Th(\cM)$ is transitive, but the converse is not necessarily true -- in fact, in Section \ref{sectionElemNonTrans} we will see examples of elementarily indivisible structures which are not transitive.
Having said that, we do have:

\begin{remark}
	If $\cM$ is homogeneous, then $\cM$ is transitive iff $\Th(\cM)$ is transitive.
\end{remark}

\begin{corollary}\label{ElemHomIsTrans}
	Every homogeneous elementarily indivisible structure is transitive. \qed
\end{corollary}

\section{Elimination of quantifiers}\label{sectionQE}

In this paper  we stick to the definition of quantifier elimination presented in \cite{Marker02}:
\begin{definition}
	We say that an $\cL$-theory $T$ admits \emph{quantifier elimination} (QE) if for every $\cL$-formula $\phi$ there is a quantifier-free $\cL$-formula $\psi$ such that
	\[ T\models \phi \leftrightarrow \psi .\] We say an $\cL$-structure  $\cM$ admits QE if $\Th(\cM)$ admits QE.
\end{definition}

\begin{remark}\label{QEmodelComplete}
	If $T$ admits QE and $\cN,\cM\models T$, $\cN\subseteq \cM$ then $\cN\preceq\cM$. So if $\cM$ is indivisible and admits QE, then it is elementarily indivisible.
\end{remark}
Furthermore: this remark can be extended to any infinitary logic. For simplicity, we restrict ourselves to $\cL_{\omega_1,\omega}$:

\begin{definition}
	We say that an $\cL_{\omega_1,\omega}$-theory $T$ admits \emph{quantifier elimination} (QE) if for every $\cL_{\omega_1,\omega}$-formula $\phi$ there is a quantifier-free $\cL_{\omega_1,\omega}$-formula $\psi$ such that
	\[ T\models \phi \leftrightarrow \psi . \] We say an $\cL$-structure  $\cM$ admits $\cL_{\omega_1,\omega}$-QE if its $\cL_{\omega_1,\omega}$-theory admits QE.
\end{definition}

\begin{remark}\label{indivisibleHom}
	It is an easy exercise to verify that for every countable structure $\cM$ in a countable relational language, $\cM$ is ultrahomogeneous iff $\cM$ admits $\cL_{\omega_1,\omega}$-QE, which, in turn, implies that every embedding is elementary. So we have that every indivisible ultrahomogeneous countable structure is elementarily indivisible.
\end{remark}

In view of the remarks above, we present below an example from \cite[Corollary 6.6]{HKO11} of an elementarily indivisible structure which does not admit QE and is not ultrahomogeneous, in a finite language.

\begin{example}\label{ex:EInotQE}
	Let $\cG$ be the colorful graph described in Example \ref{ex:Colorful}. Let $\chi:[\cG]^2\to \omega$ be the function taking each edge $\{a,b\}\in [\cG]^2$ to its color. Let $\cG^{<}$ be $\cG$ expanded by a quaternary relation symbol $R$ such that $\cG^<\models R(a,b,c,d)$ iff $\chi(a,b)<\chi(c,d)$. Notice that $R$ is definable in $\cL_{\omega_1,\omega}$ by the formula $\bigvee_{i<j\in \omega}R_i(a,b)\land R_j(c,d)$. Since $\cG$ is indivisible, so is $\cG^<$. By Remark \ref{indivisibleHom}, since $\cG$ is ultrahomogeneous, $\cG^<$ is also ultrahomogeneous and thus elementarily indivisible. Now $\cG^< \upharpoonright \{R\}$ is also elementarily indivisible and $\cG^< \upharpoonright \{R\}$ does not eliminate quantifiers; for example $\exists z,w R(z,w,x,y)$ cannot be eliminated.

\end{example}

\begin{notation}
For a set of $\cL$-structures $S$ and an $\cL$-theory $T$ we denote $S\models T$ if $\cM\models T$ for all $\cM\in S$.
\end{notation}

\subsection{The elimination}\label{theElimination}

In this subsection we prove the following theorem which is the main result of the section.

\begin{theorem}\label{compositionQE}
	Let $\cL$ be a relational language and let $T_1, T_2$ be $\cL$-theories, not necessarily complete.
	If  $T_1$ and $T_2$ both admit QE and $T_1$ is transitive then
	there is an $\cL \cup \{s\}$-theory $T$ (not necessarily complete) admitting QE, such that $\gcompMN^s \models T$ whenever $\cM \models T_1$ and $\{\cN_a\}_{a\in \cM} \models T_2$. 
	\bigskip
	
	In particular, if $\cM$ and $\cN$ are $\cL$-structures both admitting QE and $\Th(\cM)$ is transitive
	then $\cM[\cN]^s$ admits QE.
\end{theorem}

Before proving this theorem, we note that the requirement of transitivity is necessary and provide a simple example in which $\cM$ and $\cN$ both admit QE, but 
$\scomp{\cM}{\cN}^s$ does not:

\begin{example}
	Let  $\cL:=\Set{R,A,B}$ where $R$ is a binary relation and $A$, $B$ are unary predicates. Let $\cM$ be an $\cL$-structure satisfying:
	\begin{itemize}
		\item $\left| A^{\cM} \right| = 1$, $\left| B^{\cM} \right| = \aleph_0$.
		\item $A^{\cM}\cap B^{\cM} = \emptyset$.
		\item $R^{\cM}:=\Set{(a,b)| a\in A^\cM,\ b\in B^\cM}$.
	\end{itemize}
	Let $\cN$ be an $\cL$-structure with a countably infinite universe interpreting all relations in $\cL$ as empty. Then $\cM$ and $\cN$ both admit QE but $\scomp{\cM}{\cN}^s$ does not admit QE.
	Obviously $\cN$ admits QE. 
	
	It is also obvious that $\cM \upharpoonright \{A,B\}$ admits QE (where $\cM \upharpoonright \{A,B\}$ is the restriction of $\cM$ to the language $\{A,B\}$). To show $\cM$ admits QE, we note that $R^{\cM}$ is quantifier free $\emptyset$-definable from $\{A^{\cM}, B^\cM\}$.
	
	To show $\scomp{\cM}{\cN}^s$ does not admit QE, by Remark \ref{unary}, $\scomp{\cM}{\cN}^s \models U((x,y)) \iff \cN\models U(y)$ for every unary predicate $U\in \cL$. Since $\cN$ interprets all relations in $\cL$ as empty, $\cM[\cN]^s$ interprets all unary predicates as empty. Thus every quantifier-free formula in one variable is equivalent to either ``$x=x$'' or ``$x\neq x$''. Let $\phi(x):=\exists y\, R(x,y)$. Notice that $\cM[\cN]^s\models \phi((a,c))$ for $a\in A^{\cM}$ and  $\cM[\cN]^s\not\models\phi((b,c))$ for $b\in B^{\cM}$. So $\phi(x)$ is neither equivalent to ``$x=x$'' nor to ``$x\neq x$'' and thus $\cM[\cN]^s$ does not admit QE.
	
\end{example}

We continue with a few definitions and lemmas needed for the proof of Theorem \ref{compositionQE}.

Throughout this section, we use the following abbreviations:
\begin{notation}\ 
	\begin{itemize}
		\item $\bar{v}:= \left(v_1,\dots, v_n\right)$ is an $n$-tuple of variables.
		\item $\widebar{(a,b)}:=\big(a_1,b_1),\dots, (a_n, b_n)\big)$ is an $n$-tuple of elements in the product (generalized or not).
		\item Whenever $\widebar{(a,b)}=\big((a_1,b_1),\dots, (a_n, b_n)\big)$, we denote $\bar{a}:=(a_1,\dots, a_n)$ and $\bar{b}:=(b_1,\dots, b_n)$.
	\end{itemize}
\end{notation}
So whenever $\bar{v}$ and $\widebar{(a,b)}$ appear together, they are of the same length and, unless otherwise specified, we denote their length by $n$.

\begin{notation} Let $\phi$ be an $\cL$-formula. We denote $\widetilde{\phi}$ the $\cL\cup \{s\}$-formula obtained from $\phi$ by replacing the equality symbol `$=$' with $s$, namely:
	\begin{itemize}
		\item If $\phi$ is atomic of the form $R\left(\bar{v}\right)$ for $R\in\cL$, then $\widetilde{\phi} := \phi$.
		\item If $\phi$ is atomic of the form ``$x=y$'', then $\widetilde{\phi} := s(x,y)$.
		\item If $\phi$ is of the form $\alpha \ast \beta$ where $\ast \in \Set{\land, \lor, \rightarrow}$,  then $\widetilde{\phi}:= \widetilde{\alpha} \ast \widetilde{\beta}$.
		\item If $\phi$ is of the form $\neg \beta$,  then $\widetilde{\phi}:= \neg \widetilde{\beta}$.
		\item If $\phi$ is of the form $\ast x\, \beta$ where $\ast \in \Set{\forall, \exists}$  then $\widetilde{\phi}:= \ast x\, \widetilde{\beta}$.
	\end{itemize}
	
\end{notation}

\begin{lemma}\label{MNiffN} Let $\cM, \{\cN_a\}_{a\in \cM}$ be $\cL$-structures.
	If $\phi\left(\bar{v}\right)$ is a quantifier-free $\cL$-formula, $a\in \cM,\ b_1,\dots, b_n\in \cN_a$, then
	\[\cM[\cN_a]_{a\in \cM} \models \phi\left((a,b_1),\dots,(a,b_n)\right) \iff \cN_a\models\phi(b_1,\dots,b_n).\]
\end{lemma}

\begin{proof}
	
	Define $e_a:\cN_a\to \scomp{\cM}{\cN_a}_{a\in \cM}$ by $e_a(b):=(a,b)$.
	By definition of $\gcompMN$ this is an $\cL$-embedding and thus the claim follows.
\end{proof}

\begin{definition} Let $\cM, \{\cN_a\}_{a\in M}$ be $\cL$-structures.
	Let $\phi\left(\bar{v}\right)$ be an $\cL$-formula, $\widebar{(a,b)} \in \gcompMN$. We say $\widebar{(a,b)}$ is an \emph{admissible assignment} for $\phi$ when for all $R \in \cL$,  if $R(v_{i_1},\dots, v_{i_k})$ occurs in $\phi$, then $\bigvee_{1\leq l<m\leq k} a_{i_l}\neq a_{i_m}$.
\end{definition}

\begin{lemma}\label{MNiffM}
	If $\phi\left(\bar{v}\right)$ is a quantifier-free $\cL$-formula and $\widebar{(a,b)} \in \gcompMN$ is an admissible assignment for $\phi$,
	then \[\gcompMN^s \models \widetilde{\phi}\left(\widebar{(a,b)}\right) \iff \cM \models \phi\left(\bar{a}\right).\]
	
	In particular, 
	If $\phi\left(\bar{v}\right)$ is a quantifier-free $\cL$-formula such that the equality symbol does not occur in $\phi$,
	then \[\gcompMN^s \models \phi\left(\widebar{(a,b)}\right) \Leftrightarrow \cM \models\phi\left(\bar{a}\right) .\]
	
\end{lemma}

\begin{proof}\ 
	\begin{itemize}
		\item If $\phi$ is of the form ``$v_1=v_2$'', this follows by definition of $s$.
		\item If $\phi$ is of the form $R(v_{i_1},\dots,v_{i_k})$, since $\widebar{(a,b)}$ is an admissible assignment, by definition of $\gcompMN^s$, \[\gcompMN^s \models \phi\left(\widebar{(a,b)}\right) \iff \cM \models \phi\left(\bar{a}\right)\] and $\widetilde{\phi} = \phi$.
		\item For a general quantifier-free $\phi$ the claim follows by induction on the complexity of $\phi$.
	\end{itemize}
\end{proof}

\begin{definition}
	A formula $\phi\left(\bar{v}\right) $ is called a \emph{complete equality diagram} if it is a consistent conjunction of formulas of the form ``$x=y$'' and ``$x\neq y$'' such that for all $1\leq i,j\leq n$, either $\phi\left(\bar{v}\right) \vdash v_i=v_j$ or $\phi\left(\bar{v}\right) \vdash v_i\neq v_j$.
\end{definition}

\begin{lemma}\label{MNiffMextended} 
	Let $T$ be a transitive theory. For every quantifier-free $\cL$-formula $\vphi\left(\bar{v}\right)$ there is a quantifier-free $\cL \cup \{s\}$-formula $\vphi^{\prime}\left(\bar{v}\right)$ such that if
	$\cM \models T, \{\cN_a\}_{a\in \cM} $ are $\cL$-structures, $\widebar{(a,b)}\in \gcompMN^s$, then
	\[ \cM\models \vphi\left(\bar{a}\right) \iff \gcompMN^s\models \vphi^{\prime}\left(\widebar{(a,b)}\right). \]
\end{lemma}
\begin{proof} Let $\{\psi_j\}_{j\in J}$ be all complete equality diagrams on $\bar{v}$. Notice that
	\[
	\vdash \vphi\left(\bar{v}\right)\leftrightarrow \left( \vphi\left(\bar{v}\right)\land \bigvee_{j\in J} \psi_j\left(\bar{v}\right) \right)  
	\leftrightarrow \bigvee_{j\in J} \left( \vphi\left(\bar{v}\right)\land  \psi_j\left(\bar{v}\right) \right) .
	\]
	So by dealing with each disjunct separately, it suffices to find a quantifier-free $\cL\cup \{s\}$-formula $\vphi^{\prime}\left(\bar{v}\right)$ such that
	for all $\cL$-structures
	$\cM, \{\cN_a\}_{a\in \cM} $ such that $\cM\models T$ and for all $\widebar{(a,b)}\in \gcompMN^s$,
	\[ \cM\models \vphi\left(\bar{a}\right)\land \psi\left(\bar{a}\right)  \iff \gcompMN^s\models \vphi^{\prime}\left(\widebar{(a,b)}\right) \]
	where $\psi$ is a complete equality diagram.
	
	Next, for every $v_j,v_k$ such that $j<k$ and $\psi\left(\bar{v}\right)\vdash v_j = v_k$, we can replace every occurrence of $v_k$ with $v_j$, so we may assume 
	$\psi\left(\bar{v}\right) = \bigwedge_{1\leq j<k \leq n} v_j\neq v_k$. Secondly, since $T$ is transitive, every formula of the form $R(x,\dots,x)$ is equivalent either to ``$x=x$'' or to ``$x\neq x$'', so we may assume there are no such occurrences in $\vphi$. Let $\widetilde{\psi}, \widetilde{\vphi}$ be the formulas obtained from 
	$\psi,{\vphi}$ respectively, by replacing `$=$' with $s$. We claim that for all $\cL$-structures
	$\cM \models T, \{\cN_a\}_{a\in \cM} $ and $\widebar{(a,b)}\in \gcompMN^s$,
	\[ \cM\models \vphi\left(\bar{a}\right)\land \psi\left(\bar{a}\right)  \iff \gcompMN^s\models \widetilde{\vphi}\left(\widebar{(a,b)}\right) \land \widetilde{\psi}\left(\widebar{(a,b)}\right) .\]
	
	Indeed:
	by definition,  $\cM\models \psi\left(\bar{a}\right) \iff \gcompMN^s\models \widetilde{\psi}\left(\widebar{(a,b)}\right)$. Assuming $\cM\models \psi\left(\bar{a}\right)$, since there are no occurrences of the form  $R(x,\dots,x)$ in $\vphi$, $\widebar{(a,b)}$ is an admissible assignment for $\vphi$. So by Lemma \ref{MNiffM},  \[\cM\models \vphi\left(\bar{a}\right) \iff \gcompMN^s\models \widetilde{\vphi}\left(\widebar{(a,b)}\right). \]
\end{proof}

Before continuing to the main proof -- one last definition, that stand at the core of the proof of Theorem \ref{compositionQE}:

\begin{definition}
	An $\{s\}$-formula $\phi\left(\bar{v}\right)$ is called a \emph{complete $s$-diagram} if it is a conjunction of formulas of the form $s(x,y)$ or $\neg s(x,y)$ such that for every $1\leq i,j\leq n$, either $\phi(v_1,\dots,v_n) \vdash s(v_i,v_j)$ or $\phi\left(\bar{v}\right) \vdash \neg s(v_i,v_j)$ and $\phi$ is consistent with $s$ being an equivalence relation.
	
	Notice that $\phi$ is a complete $s$-diagram iff it is of the form $\widetilde{\psi}$ for some complete equality diagram $\psi$.
\end{definition}

\begin{proof}[Proof of Theorem \ref{compositionQE}]
	We provide a technical proof, noting that this proof is in fact constructive, using the elimination of quantifiers from $T_1$ and $T_2$.
	
	Let $\phi = \exists w \bigwedge_{i\in I} \theta_i\left(\bar{v},w\right) $ such that $\{\theta_i\}_{i\in I}$ are atomic and negated atomic formulas.
	We need to find a quantifier-free $\cL \cup \{s\}$-formula $\varphi$ such that whenever $\cM\models T_1$ and $\{\cN_a\}_{a\in \cM}\models T_2$, \[\gcompMN^s \models \phi\left(\bar{v}\right) \leftrightarrow \varphi\left(\bar{v}\right).\]
	
	First, since $\vdash \exists w \big( \chi\left(\bar{v},w\right) \land \theta\left(\bar{v}\right) \big) \leftrightarrow \exists w \big( \chi\left(\bar{v},w\right)\big)  \land \theta\left(\bar{v}\right)$ we may assume that $w$ occurs in $\theta_i$ for all $i\in I$.

	In order to proceed with the proof we will use complete $s$-diagrams, in a way similar to the way complete equality diagrams were used in the proof of Lemma \ref{MNiffMextended}:

	Let $T_{equiv}$ be the $\{s\}$-theory stating that $s$ is an equivalence relation and let $\{\widetilde{\psi}_j\}_{j\in J}$ be all the complete $s$-diagrams on $\bar{v},w$. There are finitely many such and
	\[T_{equiv} \models \exists w \bigwedge_{i\in I} \theta_i\left(\bar{v},w\right) \leftrightarrow  \exists w \big( \bigwedge_{i\in I} \theta_i\left(\bar{v},w\right) \land \bigvee_{j\in J} \widetilde{\psi}_j\left(\bar{v},w\right)\big)\]
	\[\leftrightarrow \exists w  \bigvee_{j\in J} \big( \widetilde{\psi}_j\left(\bar{v},w\right) \land \bigwedge_{i\in I} \theta_i\left(\bar{v},w\right) \big)\]
	\[\leftrightarrow \bigvee_{j\in J}  \exists w  \big( \widetilde{\psi}_j\left(\bar{v},w\right) \land \bigwedge_{i\in I} \theta_i\left(\bar{v},w\right) \big) .\]
	Since $\gcompMN^s \models T_{equiv}$ for every $\cM$ and $\{\cN_a\}_{a\in \cM}$, we may assume $\phi$ is of the form $ \exists w  \big( \widetilde{\psi}\left(\bar{v},w\right) \land \bigwedge_{i\in I} \theta_i\left(\bar{v},w\right) \big)$ where $\widetilde{\psi}$ is a complete $s$-diagram, $\theta_i$ are atomic and negated atomic formulas such that $w$ occurs in each $\theta_i$.
	
	Next, let $$I_2:=\Set{i\in I | \widetilde{\psi} \vdash s(v,w) \text{ for all } v\text{ occuring in }\theta_i}$$
	$$I_1:=\Set{i\in I | \widetilde{\psi} \vdash \neg s(v,w) \text{ for some } v\text{ occuring in }\theta_i} = I\setminus I_2$$
	and separate $\bar{v}$ to $\bar{v}^1, \bar{v}^2$, where $\bar{v}^2$ are the variables occurring in $\bigwedge_{i\in I_2} \theta_i\left(\bar{v},w\right)$ and $\bar{v}^1$ the ones not occurring there.
	So $\phi$ is of the form
	\[ \exists w  \left( \widetilde{\psi}\left(\bar{v}^1, \bar{v}^2 ,w\right) \land \bigwedge_{i\in I_1} \theta_i\left(\bar{v}^1, \bar{v}^2,w\right)\land \bigwedge_{i\in I_2} \theta_i\left(\bar{v}^2,w\right) \right)\]
	where $\widetilde{\psi}$ is a complete $s$-diagram. We may further assume `$=$' and $s$ do not occur in $\bigwedge_{i\in I_1} \theta_i\left(\bar{v}^1, \bar{v}^2,w\right)$, for such an occurrence would be either superfluous with respect to $\widetilde{\psi}$ or inconsistent with $\widetilde{\psi}$.

		If $\bar{v}^1 = \left(v_1^1,\dots, v^1_{n_1}\right), \bar{v}^2 = \left(v_1^2,\dots, v^2_{n_2}\right)$, let
		\begin{align*}
		& \bar{a}^1 = \left(a_1^1,\dots, a_{n_1}^1\right) , \ 
		\bar{b}^1 = \left(b_1^1,\dots, b_{n_1}^1\right) \\
		& \bar{a}^2 = \left(a_1^2,\dots, a_{n_2}^2\right) , \ 
		\bar{b}^2 = \left(b_1^2,\dots, b_{n_2}^2\right)
		\end{align*}
		and denote
		\begin{align*}
		& \widebar{(a,b)}^1 := \left(\left(a_1^1,b_1^1\right),\dots,\left(a_{n_1}^1,b_{n_1}^1\right) \right) \ 
		& \widebar{(a,b)}^2 := \left(\left(a_1^2,b_1^2\right),\dots,\left(a_{n_2}^2,b_{n_2}^2\right) \right). \ 
		\end{align*} 
		\pagebreak
		
		\noindent Claim. The following are equivalent:
		\begin{enumerate}[{label=(\arabic*)}]
			\item  \begin{align*} & \gcompMN^s \models \\ & \exists w  \left( \widetilde{\psi}\left(\widebar{(a,b)}^1, \widebar{(a,b)}^2,w\right) \land \bigwedge_{i\in I_1} \theta_i\left(\widebar{(a,b)}^1, \widebar{(a,b)}^2,w\right)\land \bigwedge_{i\in I_2} \theta_i\left(\widebar{(a,b)}^2,w\right) \right) \end{align*}
			\item There is an $a\in \cM$ such that: $a_{j}^2 = a$ for all $1\leq j\leq n_2$,
			\[ \cM \models \exists w  \left( {\psi}\left(\bar{a}^1,\bar{a}^2,w\right) \land \bigwedge_{i\in I_1} \theta_i\left(\bar{a}^1,\bar{a}^2,w\right)\right)  \text{ and } \cN_a \models \exists w\left( \bigwedge_{i\in I_2} \theta_i\left(\bar{b}^2,w\right)\right) \]
			
		\end{enumerate}

	 \begin{sproof} \ \\
			 \noindent $(\Rightarrow)$ Let $c \in \cM$ and $d\in \cN_c$ such that
			
			\begin{align*} & \gcompMN^s \models \\ & \widetilde{\psi}\left(\widebar{(a,b)}^1, \widebar{(a,b)}^2,(c,d)\right) \land  \bigwedge_{i\in I_1} \theta_i\left(\widebar{(a,b)}^1, \widebar{(a,b)}^2,(c,d)\right)\land \bigwedge_{i\in I_2} \theta_i\left(\widebar{(a,b)}^2,(c,d)\right) . \end{align*}
			
			 By definition, $\cM \models \psi\left(\bar{a}^1,\bar{a}^2,c\right)$, and since $\psi\left(\bar{a}^1, \bar{a}^2,c\right)$ implies that \\  $\widebar{(a,b)}^1, \widebar{(a,b)}^2,(c,d)$ is an admissible assignment for $\bigwedge_{i\in I_1} \theta_i\left(\bar{v}^1, \bar{v}^2,w\right)$, by Lemma \ref{MNiffM},
			\[\cM \models  \bigwedge_{i\in I_1} \theta_i\left(\bar{a}^1, \bar{a}^2, c\right) .\]
			
			Furthermore, by the definition of $I_2$, \[\psi\left(\bar{a}^1,\bar{a}^2,c\right) \vdash\left( \bigwedge_{1\leq j\leq n_2} a_j^2 = c\right) \land \left( \bigwedge_{1\leq j,k\leq n_2} a_j^2 =  a_k^2 \right).\]
			So letting $a:=c$, in fact
			\[\gcompMN^s \models \bigwedge_{i\in I_1} \theta_i\left(\left(a,b_1^2\right),\dots, \left(a,b_{n_2}^2\right) , (a,d)\right)  \]
			so by Lemma \ref{MNiffN},
			\[\cN_a \models \bigwedge_{i\in I_2} \theta_i\left(\bar{b}^2,d\right)  .\]
			
			\noindent $(\Leftarrow)$ Let $a\in \cM$ be such that for all $1\leq j\leq n_2$, $a_{j_2}^2 = a$, and let $c\in \cM$ and $d\in \cN_a$ be such that
			\[\cM \models  {\psi}\left(\bar{a}^1,\bar{a}^2,c\right) \land \bigwedge_{i\in I_1} \theta_i\left(\bar{a}^1, \bar{a}^2, c\right)  \text{ and  }\ \cN_a \models   \bigwedge_{i\in I_2} \theta_i\left(\bar{b}^2, d\right). \]
			Since ${\psi}\left(\bar{a}^1,\bar{a}^2,c\right) \vdash\bigwedge_{1\leq j\leq n_2} a_j^2 = c $, in fact $c = a$, so $d\in \cN_c$ and 
			\[\cN_c \models   \bigwedge_{i\in I_2} \theta_i\left(\bar{b}^2, d\right), \] thus by Lemma \ref{MNiffN},
			\[\gcompMN^s \models  \bigwedge_{i\in I_2} \theta_i\left(\widebar{(a,b)}^2, (c,d)\right)  .\]
			Since  ${\psi}\left(\bar{a}^1,\bar{a}^2,c\right)$ implies that $\widebar{(a,b)}^1, \widebar{(a,b)}^2,(c,d)$ is an admissible assignment for $\phi_1$, by Lemma \ref{MNiffM}, \[\gcompMN^s \models  \widetilde{\psi}\left(\widebar{(a,b)}^1, \widebar{(a,b)}^2,(c,d)\right) \land   \bigwedge_{i\in I_1} \theta_i\left(\widebar{(a,b)}^1, \widebar{(a,b)}^2, (c,d)\right). \]

	\end{sproof}
	
	Assuming $\cM \models T_1$ and  $\{\cN_a\}_{a\in \cM}\models T_2$, by QE of $T_1$ and $T_2$, let $\vphi_1\left(\bar{v}\right), \vphi_2\left(\bar{v}\right)$ be quantifier-free $\cL$-formulas such that
	\[T_1\models \exists w  \left( \widetilde{\psi}\left(\bar{v}^1,\bar{v}^2,w\right) \land \bigwedge_{i\in I_1} \theta_i\left(\bar{v}^1,\bar{v}^2,w\right)\right) \leftrightarrow \vphi_1\left(\bar{v}^1,\bar{v}^2\right) \text{\ \  and }\]\[ T_2 \models \exists w\left( \bigwedge_{i\in I_2} \theta_i\left(\bar{v}^2,w\right)\right)\leftrightarrow \vphi_2\left(\bar{v}^2\right) .\]
	
	So (2) from the claim is equivalent to:
	
	\begin{enumerate}[{label=(\arabic*)}]
		\setcounter{enumi}{2}
		\item  There is an $a\in \cM$ such that: $a_{j}^2 = a$ for all $1\leq j\leq n_2$, 
		\[ \cM \models  \vphi_1\left(\bar{a}^1,\bar{a}^2\right) \text{ and } \cN_a \models  \vphi_2\left(\bar{b}^2\right) .\]
	\end{enumerate}

	By Lemmas \ref{MNiffN} and \ref{MNiffMextended}, there is an $\cL$-formula $\vphi_1^{\prime}$ such that (3) above is equivalent to: 
	
	\begin{enumerate}[{label=(\arabic*)}] \setcounter{enumi}{3}
		\item $\bigwedge_{ 1\leq j,k\leq n_2 } a_j^2 = a_k^2$ and
		\[ \gcompMN^s \models  \vphi_1^{\prime}\left(\widebar{(a,b)}^1, \widebar{(a,b)}^2\right) \text{ and } \gcompMN^s \models  \vphi_2\left(\widebar{(a,b)}^2\right) ,\]
		which, in turn, is equivalent to
		\item  \begin{align*} & \gcompMN^s \models \\ & \left(\bigwedge_{1\leq j,k\leq n_2} s\left(\left(a_j^2,b_j^2\right),\left(a_k^2,b_k^2\right)\right)\right) \land \vphi_1^{\prime}\left(\widebar{(a,b)}^1, \widebar{(a,b)}^2\right) \land \vphi_2\left(\widebar{(a,b)}^2\right) .\end{align*}
	\end{enumerate}

	Setting \[\vphi\left(\bar{v}^1,\bar{v}^2\right): = \left(\bigwedge_{1\leq j,k\leq n_2} s\left(v_j^2,v_k^2\right)\right) \land \vphi_1^{\prime}\left(\bar{v}^1, \bar{v}^2\right) \land \vphi_2\left(\bar{v}^2\right), \]
	we get that for all $\widebar{(a,b)}^1\in \left(\gcompMN^s \right)^{n_1}$ and $\widebar{(a,b)}^2 \in \left(\gcompMN^s \right)^{n_2}$:
	\[\gcompMN^s \models \phi \left(\widebar{(a,b)}^1,\widebar{(a,b)}^2\right) \iff \gcompMN^s \models \vphi \left(\widebar{(a,b)}^1,\widebar{(a,b)}^2\right). \]
	So
	\[ \gcompMN^s \models \phi\left(\bar{v}^1,\bar{v}^2\right) \leftrightarrow \vphi\left(\bar{v}^1,\bar{v}^2\right)\]
	and $\vphi$ is quantifier-free.
	
	Let $\vphi_{\phi}$ be the quantifier-free $\cL \cup \{s\}$-formula obtained from $\phi$ by the above process. Let $T$ be the logical closure (all the logical consequences) of
	\[ T_{equiv} \cup \Set{\phi \leftrightarrow \vphi_{\phi} | \phi \text{ is of the form } \exists w  \left( \widetilde{\psi}\left(\bar{v},w\right) \land \bigwedge_{i\in I} \theta_i\left(\bar{v},w\right) \right) }. \]
	
	$T$ admits QE and by the above process, $\gcompMN^s\models T$ for all $\cM\models T_1$, $\{\cN_a\}_{a\in \cM}\models T_2$.
\end{proof}

Note that in the proof above, transitivity of $T$ is used to get from (3) to (4) as $\vphi_1$ can include occurrences of the form $R(x,\dots, x)$ that would be interpreted in the product differently in each copy of $\cN_a$, and in general we cannot use Lemma \ref{MNiffMextended} if $\Th(\cM)$ is not transitive.

We note that if $T_1$ and $T_2$ are complete, so is $T$ and thus:

\begin{corollary}
	If $\cM_1\equiv \cM_2$ and $\Set{\cN_a | a\in \cM_1\cup \cM_2}$ are pairwise elementarily equivalent then $\cM_1[\cN_a]_{a\in \cM_1}\equiv \cM_2[\cN_a]_{a\in \cM_2}$.

\end{corollary}

We leave it as an exercise to show that $s$ is necessary; i.e. find $\cL$-structures $\cM$ and $\cN$ (even elementarily indivisible), such that $\cM$ and $\cN$ both admit QE but $\cM[\cN]$ does not (not even model complete).
\bigskip

\subsection{Application to elementary indivisibility}\label{applyQE}

In this subsection, we provide an immediate application of Theorem \ref{compositionQE} to elementary indivisibility, mainly proving that the lexicographic product of two elementarily indivisible  structures is elementarily indivisible. Here we only use the result of QE for $\scomp{\cM}{\cN}^s$, though in the following sections the full power of Theorem \ref{compositionQE} regarding the generalized product will be needed.

\begin{definition}
Let $\cM$ and $\cM'$ be structures with the same universe, not necessarily in the same language.

We say $\cM'$ is a \emph{language reduct} of $\cM$ if $\cM' = \cM \upharpoonright\cL_0$ for some $\cL_0\subseteq \cL$.

We say $\cM'$ is a \emph{definitional reduct} of $\cM$ if every $\emptyset$-definable relation in $\cM'$ is $\emptyset$-definable in $\cM$.
\end{definition}

\begin{notation}\label{defMorleyzation}
	Let $\closure{\cL}$ be an expansion of $\cL$ such that for each $\cL$ -formula $\phi\left(\bar{v}\right)$ with $n$ free variables, we add an $n$-ary relation $R_{\phi}$ (we denote by $\phi_R\left(\bar{v}\right)$ the formula that defined $R$).
	
	For any $\cL$-structure $\cM$, we define $\closure{\cM}$ an  $\closure{\cL}$ -structure whose universe is the universe of $\cM$ , and for every $n$-ary relation symbol $R \in \closure{\cL}$ we set
	$$ R^{\closure{\cM}} =
	\left\{
	\begin{array}{ll}
	R^{\cM}  & \mbox{if } R \in \cL \\
	\{\ \bar{a} \in \cM^n\ |\ \cM \models \phi_R\left(\bar{a}\right) \ \} & \mbox{if } R \in \closure{\cL} \setminus \cL 
	\end{array}
	\right. $$
	We call $\closure{\cM}$ the Morleyzation of $\cM$. It is well-known and an easy exercise to verify that $\widehat{\cM}$ admits QE.
\end{notation}

We note that while it is obvious that if $\cM$ is (elementarily) indivisible and $\cM'$ is a language reduct of $\cM$, then $\cM'$ is also (elementarily) indivisible, this is not true for definitional reducts. For example consider the ordered natural numbers $\langle \omega,<\rangle$.
The following lemma implies this is not the case in the \emph{elementarily} indivisible context. Because of this, and following \cite{Mac11} and the extensive study done in the subject, we use \emph{reduct} as an abbreviation for definitional reduct.

\begin{lemma}\label{TMI} Let $\cM$ be an $\cL$-structure.
	The following are equivalent:
	\begin{enumerate}[label=(\arabic*)]
		\item $\cM$ is elementarily indivisible.
		\item every reduct of $\cM$ is elementarily indivisible.
		\item every reduct of $\cM$ is indivisible. 

		\item $\widehat{\cM}$ is indivisible.
		\item $\widehat{\cM}$ is elementarily indivisible.
	\end{enumerate}
\end{lemma}

\begin{proof}\ 
	
	\begin{itemize}
		\item (2)$\Rightarrow$(3)$\Rightarrow$(4) is obvious, since $\closure{\cM}$ is a reduct of $\cM$.
		\item (4)$\Rightarrow$(5) is by quantifier elimination of the Morleyzation, and model completeness (Remark \ref{QEmodelComplete}).
		\item (5)$\Rightarrow$(1) is due to elementary indivisibility respecting language reducts.
		\item (1)$\Rightarrow$(2) Let $\cM'$ be a reduct of $\cM$ in a language $\cL'$. Let $c:\cM\to \{0,1\}$ be a coloring and let $\cN\subseteq \cM$ be a monochromatic elementary substructure isomorphic to $\cM$ with universe $N$. We will show the induced $\cL'$-substructure of $\cM'$ on $N$ is an elementary substructure isomorphic to $\cM'$. Since $\cM'$ is a reduct of $\cM$, for every $\cL'$-formula $\phi$, there is an $\cL$-formula $\vphi_{\phi}$ such that $\cM'\models \phi\left(\bar{a}\right)\iff \cM\models \vphi\left(\bar{a}\right)$ for every $\bar{a}\in \cM$. In particular, for every $R \in \cL'$, there is an $\cL$-formula $\vphi_R$ such that $\cM'\models R\left(\bar{a}\right)\iff \cM\models \vphi_R\left(\bar{a}\right)$.
		
		Let $\cN'$ be the $\cL'$-structure whose universe is $N$ and for every $R \in \cL'$ 
		\[R^{\cN'}:=\Set{\bar{a} | \cN\models \vphi_R\left(\bar{a}\right) }.\]
		Since $\cN\cong \cM$, also $\cN'\cong \cM'$. Since $\cN\prec\cM$, for every $R\in \cL'$ we have
		\[ \cN'\models R\left(\bar{a}\right)\iff \cN\models \vphi_R\left(\bar{a}\right) \iff \cM\models \vphi_R\left(\bar{a}\right) \iff \cM'\models R\left(\bar{a}\right), \]
		so, in fact, $\cN'$ coincides with the induced $\cL'$-substructure of $\cM$ on $N$. But the above equivalence can also be achieved for $\cL'$-formulas:
		\[ \cN'\models \phi\left(\bar{a}\right)\iff \cN\models \vphi_\phi\left(\bar{a}\right) \iff \cM\models \vphi_\phi\left(\bar{a}\right) \iff \cM'\models \phi\left(\bar{a}\right). \]
		So $\cN'$ is an elementary substructure of $\cM'$.
		 
	\end{itemize}
\end{proof}

The following proposition is in fact almost identical to a part of  \cite[Proposition 2.14]{HKO11}, but for the sake of completeness we give a simple proof here.

\begin{proposition}\label{compositionIndivisible}
	If $\cM$ and $\cN$ are both indivisible then so is $\cM[\cN]^s$
\end{proposition}
\begin{proof}
	Let $c:\cM[\cN]^s \rightarrow \{0,1\}$ be a coloring of $\cM[\cN]^s$.
	So for each $a \in \cM$, $c$ induces a coloring of $\{a\} \times \cN$ and $\{a\} \times \cN \cong\cN$, so $\{a\} \times \cN $ is indivisible.
	So for each $a \in M$ there is ${\cN}(a) \subseteq \{a\}\times \cN$ s.t. $| c[\cN(a)] | = 1$ and $\cN(a)\cong \cN$.
	Now, let us define a coloring $C: \cM \rightarrow \{\{0\},\{1\}\}$ as follow:
	$ C(a) := c[\cN(a)] $. From the previous statement it follows that $C$ is well-defined. So $C$ is a coloring of $\cM$ and since $\cM$ is indivisible, there is as $C$-monochromatic substructure $\cM_0 \subseteq \cM$ isomorphic to $\cM$.
	Let $\mathcal{A} \subseteq \cM[\cN]^s$ be the substructure
	\[ \bigcup_{a \in \cM_0} \cN(a).\]
	
	By its construction,  $\mathcal{A}$ is $c$-monochromatic. To show $A\cong \cM[\cN]^s$,
	let $f:\cM_0 \overset{\cong}{\to} \cM$ be an isomorphism and for every $a \in M_0$, let $g_a : \cN(a) \overset{\cong}{\to} \cN$ be an isomorphism. We define $F: \mathcal{A} \to \cM[\cN]^s$ by \[  F((a,b)) = \left(\ f(a),\ g_a((a,b))\ \right). \]
	
	We leave it to the reader to verify that $F$ is indeed an isomorphism.

\end{proof}
\begin{theorem}\label{compElemInd}
	If $\cM$ and $\cN$ are elementarily indivisible, then so are $\cM[\cN]$ and $\cM[\cN]^s$.
\end{theorem}
\begin{proof}
	
	First note that $\cM[\cN]$ is a reduct of $\cM[\cN]^s$, so it suffices to show elementary indivisibility only for $\cM[\cN]^s$.
	
	From the assumption and by Lemma \ref{TMI}, $\closure{\cM}$ and $\closure{\cN}$ are elementarily indivisible in $\closure{\cL}$, thus by Proposition \ref{compositionIndivisible}, $\closure{\cM}[\closure{\cN}]^s$ is indivisible in $\closure{\cL}$. By Lemma \ref{elemTrans}, $\Th\left(\closure{\cM}\right)$ is transitive and since $\closure{\cM}$ and $\closure{\cN}$ both admit QE in $\closure{\cL}$, by Theorem \ref{compositionQE},  $\closure{\cM}[\closure{\cN}]^s$ admits QE. In conclusion,
	$\closure{\cM}[\closure{\cN}]^s$ is indivisible  and admits QE, thus it is elementarily indivisible and $\scomp{\cM}{\cN}^s$ is a reduct of $\closure{\cM}[\closure{\cN}]^s$ to $\cL$.
\end{proof}

We observe that $\widehat{\cM}[\widehat{\cN}]^s$ is not the same structure as $\widehat {\cM[\cN]^s}$. Thus, Theorem \ref{compositionQE} does not automatically imply QE relative to $\cM$ and $\cN$, i.e., the QE assumption in statement of the theorem cannot be dropped. The following proposition remedies this situation.

\begin{proposition}
	Let $\cM, \cN$ be structures in a relational language such that $\Th(\cM)$ is transitive.
	If $\phi\left(\bar{v}\right)$ is any $\cL\cup\{s\}$-formula, then there are $\cL$-formulas $\Set{\vphi_1^j\left(\bar{v}\right), \vphi_2^j\left(\bar{v}\right)}_{j=1}^k$ such that  for every $\widebar{(a,b)}\in \scomp{\cM}{\cN}^s$:
	\[ \scomp{\cM}{\cN}^s \models \phi\left(\widebar{(a,b)}\right) \iff \bigvee_{j=1}^k \left( \cM\models \vphi_1^j\left(\bar{a}\right)\land \cN\models \vphi_2^j\left(\bar{b}\right) \right) \]
\end{proposition}

\begin{proof}
	Since $\scomp{\cM}{\cN}^s = \scomp{\widehat{\cM}}{\widehat{\cN}}^s \upharpoonright \cL \cup \{s\}$ and since $\scomp{\widehat{\cM}}{\widehat{\cN}}^s$ admits QE, there is a quantifier-free $\widehat{\cL}\cup \{s\}$-formula $\vphi\left(\bar{v}\right)$ such that 
	\[\scomp{\cM}{\cN}^s\models \phi\left(\widebar{(a,b)}\right) \iff \scomp{\widehat{\cM}}{\widehat{\cN}}^s \models \phi\left(\widebar{(a,b)}\right) \iff 
	\scomp{\widehat{\cM}}{\widehat{\cN}}^s \models \vphi\left(\widebar{(a,b)}\right) \]
	for every $\widebar{(a,b)}\in \scomp{\cM}{\cN}^s$.
	By taking the disjunctive normal form (DNF) of $\vphi\left(\bar{v}\right)$, conjuncting with the disjunction with all complete $s$-diagrams and using disjunctions, we may assume $\vphi\left(\bar{v}\right)$ is of the form $\widetilde{\psi}\left(\bar{v}\right) \land \bigwedge_{i\in I} \theta_i\left(\bar{v}\right)$
	where $\theta_i$ are atomic and negated atomic formulas.
	As in the proof of Theorem \ref{compositionQE}, there are quantifier-free $\widehat{\cL}$-formulas $\widehat{\vphi_1}\left(\bar{v}\right)$ and $\widehat{\vphi_2}\left(\bar{v}\right)$ such that
	\[ \scomp{\widehat{\cM}}{\widehat{\cN}}^s \models \vphi\left(\widebar{(a,b)}\right) \iff \widehat{\cM}\models \widehat{\vphi_1}\left(\bar{a}\right)\text{ and } \widehat{\cN}\models \widehat{\vphi_2}\left(\bar{b}\right)  \] for every $\widebar{(a,b)}\in \scomp{\cM}{\cN}^s$. Since $\widehat{\cM}$ and $\widehat{\cN}$ are reducts of $\cM$ and $\cN$ respectively, there are $\cL$-formulas $\vphi_1\left(\bar{v}\right)$ and $\vphi_2\left(\bar{v}\right)$ such that 
	\[  \widehat{\cM}\models \widehat{\vphi_1}\left(\bar{a}\right) \iff \cM\models\vphi_1\left(\bar{a}\right) \text{ and } \widehat{\cN}\models \widehat{\vphi_2}\left(\bar{b}\right) \iff 
	\cN\models \vphi_2\left(\bar{b}\right)  \]
	for every $\widebar{\left(a,b\right)}\in \scomp{\cM}{\cN}^s$.
\end{proof}

\section{The existence of non-transitive elementarily indivisible structures}\label{sectionElemNonTrans}

 In this section, we give a construction for non-transitive elementarily indivisible structures. Noting that every elementarily indivisible homogeneous structure is transitive, this gives a negative answer to Question \ref{q:elemHom}. In Subsection \ref{twoOrbits} we prove the main result of this section and in Subsection \ref{subsec:infOrb}, we generalize this result by constructing elementarily indivisible structures with infinitely many orbits. The generalization will be used in Section \ref{elemNonSym}.

\subsection{Two orbits}\label{twoOrbits}
\begin{definition}
	Let $\cL$ be a relational language, an elementarily indivisible pair in $\cL$ is a pair of elementarily indivisible $\cL$-structures $\langle \cM_0, \cM_1 \rangle$ such that $\cM_0 \prec \cM_1$ and $\cM_0\not \cong \cM_1$.
\end{definition}
The existence of such a pair is needed for our construction, and thus we hereby present two key examples. While the first one is simpler, we give the second as an example in a finite language. The following examples rely on the theory of Fra\"iss\'e, developed in \cite{Fra54} and outlined in detail in \cite[Section 7.1]{Hodges93}.

\begin{example}\label{exampleElemPair}
	Let $n\geq 2$ (it wouldn't harm to assume $n=2$).
	Let $\cL = \{R_i\}_{i\in \omega}$ where all $R_i$ are relation symbols of arity $n$.

	Let $\cC$ be the class of all finite $n$-uniform hypergraphs (for $n=2$ this is simply graphs) with edges colored in $\omega$-many colors, i.e. all finite $\cL$-structures for which each $R_i$ form the edges of an $n$-uniform hypergraph and $\{R_i\}_{i\in \omega}$ are disjoint. This is a Fra\"iss\'e class and let $\cM_1$ be its Fra\"iss\'e limit.
	
	Let $\cD\subset \cC$ be the class of all structures in $\cC$ which are \emph{complete}, i.e. satisfying the property that every subset of size $n$ form an edge (in some color $c\in \omega$). Note that $\cD$ is not an elementary class, but it is a Fra\"iss\'e subclass of $\cC$. Let $\cM_0$ be its Fra\"iss\'e limit. 
	
	By universality, $\cM_0$ embeds into $\cM_1$, so we may assume $\cM_0\subset\cM_1$. It is well known and easy to verify that $\cM_1\equiv\cM_0$ and they admit QE, so $\cM_0\prec \cM_1$ and since $\age\left(\cM_0\right)\subsetneq \age\left(\cM_1\right)$, $\cM_0\not \cong \cM_1$. 
	
	Indivisibility of both $\cM_0$ and $\cM_1$ can be shown either by generalizing the well known proof of \emph{indivisibility} (in fact of  \emph{the pigeonhole property}) of the random graph (given in  \cite[Corollary 1.5]{Hen71}) or by generalizing the proof of \cite[Example 6.3]{HKO11} (the proof for $\cD$ and $n=2$ is given, for $\cC$ and $n\geq 2$ the proof is exactly the same). \emph{Elementary indivisibility} thus follows from QE and Remark \ref{QEmodelComplete}.
\end{example}

\begin{example}\label{exampleElemPairFin}
	Let $\cL:=\{R_i\}_{i\in\omega}$ and let $\cM:=\cM_1$ be as in Example \ref{exampleElemPair}. Let $f:\cM^n\to \omega$ be such that $f\left(\bar{a}\right) = i$ iff $\cM\models R_i\left(\bar{a}\right)$. Let $S$ be a binary relation on $\omega$. Let  $R^S$ be a $2n$-ary relation symbol and let $\cM^S$ be the $\{R^S\}$-structure with the same universe as $\cM$ such that whenever $\bar{a},\bar{b}\in \cM^n$ are $n$-tuples, $\cM^S\models R^S \left(\bar{a},\bar{b}\right)$ iff $S\left(f\left(\bar{a}\right), f\left(\bar{b}\right)\right)$.
	
	Consider the case $S$ is the standard order on $\omega$. Notice that the set of $\{R^<\}$-formulas
	 \[\Delta:=\Set{\exists \bar{y}_1,\dots, \bar{y}_n \bigwedge_{i=1}^{n-1} R^<\left(\bar{y}_i,\bar{y}_{i+1}\right) | n\in \omega}\] is finitely satisfiable in $\cM^<$, but not realized in $\cM^<$. Let $\cM^<\prec \cN^<$ be a countable elementary extension realizing $\Delta$ and let $N$ be its universe.
	
	We would like to show both $\cM^<$ and $\cN^<$ are elementarily indivisible.
	 Let $E$ be the equivalence relation  on $N^n$ defined by $E\left(\bar{a};\bar{b}\right)$ iff $\cN^<\models \neg R^<\left(\bar{a},\bar{b}\right) \land \neg R^<\left(\bar{b},\bar{a}\right)$. Since $N$ is countable, we may assume $[N]^n/E = \omega$ and let $\pi:[N]^n\to \omega$ be the quotient map sending each $A\in [N]^n$ to its $E$-equivalence class. Let $\cN$ be the $\cL$-structure whose universe is $N$ such that $\cN\models R_i\left(a_1,\dots, a_n\right)$ iff $\pi\left(\Set{a_1,\dots, a_n}\right) = i$ for every $a_1,\dots, a_n\in N$. Since $\cM^<\prec \cN^<$, in particular, $\cN^<$ satisfies the following property:
	 \begin{enumerate}[label=(\Alph*)]
	 	\item for every finite $A\subset N$ and every $f:[A]^{n-1}\to N^n$, there is some $b\in N$ such that
	 	$ \cN\models E\left(\bar{a},b; f\left(\bar{a}\right)\right) $
	 	for every $\bar{a}\in A^{n-1} $.
	 \end{enumerate}  By property (A), $\cN$ is ultrahomogeneous and has the same age as $\cM$, so $\cN\cong \cM$ and thus $\cN^<\cong\cM^{\vartriangleleft}$ for some non-standard order $\vartriangleleft$ on $\omega$.
	  
	  Finally, to prove $\cM^<$ and $\cM^{\vartriangleleft}$ are both elementarily indivisible, we will show that $\cM^S$ is elementarily indivisible for every binary relation $S$ on $\omega$. For that, let 
	  $\widetilde{\cM^S}$ be $\cM$ expanded by $R^S$ defined above. Notice that $R^S$ is definable in $\cL_{\omega_1,\omega}$, so clearly $\widetilde{\cM^S}$ is indivisible. By Remark \ref{indivisibleHom}, since $\cM$ is ultrahomogeneous, so is $\widetilde{\cM^S}$ and thus $\widetilde{\cM^S}$ is elementarily indivisible. Finally $\cM^S = \widetilde{\cM^S}\upharpoonright \{R^S\}$.
	  \qed  
\end{example}

\begin{definition}
For $\cL$-structures $\cM$ and $\cN$, we denote  $\cM \sim_e \cN$ if both $\cM$ can be elementarily embedded in $\cN$ and vice-versa.
\end{definition}
\begin{lemma}\label{elemEquivResult}
	If $\cM \sim_e \cN$  then
	$\cM$ is elementarily indivisible iff $\cN$ is elementarily indivisible.
\end{lemma}
\begin{proof}
	Because $ \sim_{e} $ is an equivalence relation, it suffices to show one direction. Suppose $\cM$ is elementarily indivisible and assume, without loss of generality,  $\cM\preceq \cN$.
	Let $ c:\cN\rightarrow\{\text{red},\text{blue}\} $ be a coloring of $\cN$, so $c$ naturally induces a coloring of $\cM$. Since $\cM$ is elementarily indivisible, there is a $c$-monochromatic $\cM'\prec\cM$ such that $\cM'\cong \cM$.
	Now, since $\cN$ can be elementarily embedded in $\cM$ and $\cM'\cong\cM$, in particular, there is some $\cN_0\prec \cM'$ such that $\cN_0\cong\cN$ and since $\cM'$ is monochromatic, so is $\cN_0$.
\end{proof}

\begin{lemma}\label{automorphisInduces}
	Let $\cM$, $\{\cN_a\}_{a\in \cM}$ be $\cL$-structures. For every $\left(a_1,b_1\right),\left(a_2,b_2\right)\in \gcompMN^s$, if there is an automorphism $\sigma\in \aut{\gcompMN^s}$ such that $\sigma\left(\left(a_1,b_1\right)\right)=\left(a_2,b_2\right)$, then $\cN_{a_1}\cong \cN_{a_2}$.
\end{lemma}

\begin{proof}
	$\sigma$ sends $s$-equivalence classes to $s$-equivalence classes and $\{a\}\times \cN_a$ is an $s$-equivalence class for every $a\in M$. Therefore $\sigma[\{a_1\}\times \cN_{a_1}] = \{a_2\}\times \cN_{a_2}$,
	so $\sigma \upharpoonright \{a_1\}\times \cN_{a_1}: \{a_1\}\times \cN_{a_1} \to \{a_2\}\times \cN_{a_2}$ is an isomorphism, but $\{a_1\}\times \cN_{a_1}\cong \cN_{a_1}$ and $\{a_2\}\times \cN_{a_2}\cong \cN_{a_2}$.
\end{proof}

\begin{theorem}\label{elemNontransitive}
	Let $\cM$ be a transitive elementarily indivisible structure and $\langle \cN_0, \cN_1\rangle$ an elementarily indivisible pair. Let $\cM'\subseteq \cM_1\subset \cM$ be such that $\cM'\cong \cM$ and $\cM'\prec \cM$ and let
	\[\cN_a:= \twopartdef{\cN_1}{a\in \cM_1}{\cN_0}{a\notin \cM_1.} \]
	Then the generalized product $\gcompMN^s$ is elementarily indivisible and is not transitive.
\end{theorem}

\begin{proof}
	To prove $\gcompMN^s$ is elementarily indivisible, we may assume that $\cM, \cN_0, \cN_1$ all admit QE. If not, by looking at $\widehat{\cM}, \widehat{\cN_0}, \widehat{\cN_1} $, the assumptions remain true and we can, assuming QE, prove that $\widehat{\cM}[\widehat{\cN_a}]^s_{a\in \cM}$ is elementarily indivisible, so $\gcompMN^s$ is also elementarily indivisible, as a reduct of such.  
	
	We will show that $\gcompMN^s \sim_e \scomp{\cM}{\cN_1}^s$. By Theorem \ref{compElemInd}, $\scomp{\cM}{\cN_1}^s$ is elementarily indivisible, thus by Lemma \ref{elemEquivResult}, $\gcompMN^s$ will also be elementarily indivisible.
	
	Clearly $\scomp{\cM}{\cN_1}^s\cong \scomp{\cM'}{\cN_1}^s$ and $\scomp{\cM'}{\cN_1}^s\subset \gcompMN^s$, so there is an embedding $e_1:\scomp{\cM}{\cN_1}^s \into \gcompMN^s$; on the other hand, clearly $\gcompMN^s\subset \scomp{\cM}{\cN_1}^s$, so we have an embedding $e_2:\gcompMN^s\into \scomp{\cM}{\cN_1}^s$.
	
	Now by QE of $\cM$ and of $\cN_0 \equiv \cN_1$ and by Theorem \ref{compositionQE}, there is an $\cL \cup \{s\}$-theory $T$ admitting QE, such that $\scomp{\cM}{\cN_1}^s, \gcompMN^s \models T$. By QE of $T$, $e_1$ and $e_2$ are elementary embeddings.
\end{proof}

\begin{corollary}
	There is an elementarily indivisible structure (in a finite language) that is not transitive and not homogeneous.
\end{corollary}

\begin{proof}
	Let $\cM$ be any transitive elementarily indivisible structure (all the classic examples in the introduction are), let $\langle\cN_0, \cN_1\rangle$ be an elementarily indivisible pair (in a finite language), e.g. Example \ref{exampleElemPairFin}. Let $\cM_1\subset \cM$ be such that there is some $\cM'\subseteq \cM_1$ satisfying $\cM'\cong \cM$ and $\cM'\prec \cM$ (by elementary indivisibility, there are $2^{\aleph_0}$ such but it does not harm to assume $\cM_1$ is co-finite) and let \[\cN_a:= \twopartdef{\cN_1}{a\in \cM_1}{\cN_0}{a\notin \cM_1.} \]
	By Theorem \ref{elemNontransitive}, $\gcompMN^s$ is elementarily indivisible and not transitive. By Corollary \ref{ElemHomIsTrans}, $\gcompMN^s$ is not homogeneous.
\end{proof}
\bigskip

\subsection{Infinitely many orbits}\label{subsec:infOrb}
In this subsection, we generalize the result from Subsection \ref{twoOrbits} and prove the existence of an elementarily indivisible structure with infinitely many orbits. We will use such a structure in Section \ref{elemNonSym}. Here, by an \emph{orbit} of a structure, we mean an orbit of the action of its automorphism group on its universe, i.e:

\begin{definition}
Let $\cM$ be an $\cL$ structure with universe $M$. An \emph{orbit} of $\cM$ is a set of the form $\Set{\sigma(a) | \sigma\in \aut(\cM)}$
where $a \in \cM$ (i.e. $a$ is a singleton).
\end{definition}

For the construction, we need  the following.

\begin{lemma}\label{infChain}
	There is an infinite set of elementarily indivisible pairwise-non-isomorphic structures $\{\cA_i\}_{i\in\omega}$, such that $\cA_i\sim_e \cA_j$ for all $i,j\in \omega$. Furthermore, $\{\cA_i\}_{i\in\omega}$ can be chosen to be in a finite language.
\end{lemma}

\begin{proof}
	Let $\cM$ be a transitive elementarily indivisible structure and $\langle \cN_0, \cN_1\rangle$ an elementarily indivisible pair. Without loss of generality, they all admit QE. Let $\cM\supseteq \cM_0 \supsetneq \cM_1 \supsetneq \cM_2 \supsetneq \dots$ be an infinite descending chain of substructures satisfying the following:
	\begin{itemize}
	\item $\cM$ can be embedded into $\cM_i$ for every $i\in\omega$.
	\item For every $0\leq i<j\leq \omega$, either $\cM_i \not\cong \cM_j$ or $\cM\setminus \cM_i \not\cong \cM\setminus \cM_j$.
	\end{itemize}
	By induction and indivisibility of $\cM$, given $\cM_i$, there are many appropriate choices for $\cM_{i+1}$ (though there is no harm in assuming $\cM_0=\cM_1$ and $\cM_{i+1}$ is just a co-finite substructure of $\cM_i$).
	
	For every $i\in \omega$ and $a\in M$, denote
	\[\cN_a^i:=\twopartdef{\cN_1}{a\in \cM_i}{\cN_0}{a\notin \cM_i}\]
	and let $\cA_i:= \cM[\cN_a^i]_{a\in M}^s$.
	Clearly $\scomp{\cM}{\cN_1}^s \supseteq \cA_0 \supset \cA_1 \supset \cA_2\supset ...$ and they are pairwise-non-isomorphic. Since $\cM$ can be embedded into each $\cM_i$, and each $\cA_i$ embeds $\cM_i[\cN_1]^s$, it follows that $\cM[\cN_1]^s$ can be embedded into each $\cA_i$. Each $\cA_i$ can be embedded into $\scomp{\cM}{\cN_1}^s$ via the inclusion map.  By Theorem \ref{compositionQE} these embeddings are elementary, so  $\cA_i\sim_e \scomp{\cM}{\cN_1}^s$ for every $i\in \omega$. By Theorem \ref{compElemInd} the latter is elementarily indivisible and thus by Lemma \ref{elemEquivResult} so are all $\cA_i$. If we choose $\langle \cN_0,\cN_1\rangle$ to be in a finite language, then $\{\cA_i\}_{i\in\omega}$ will also be in a finite language. 
\end{proof}

\begin{theorem}\label{infOrb}
	Let $\{A_i\}_{i\in \omega}$ be as in Lemma \ref{infChain} and let $\cM$ be an elementarily indivisible structure. If $\{\cN_a\}_{a\in \cM}$ is a collection of structures satisfying
	\[\Set{\cN_a | a\in \cM} = \Set{\cA_i | i\in \omega}\] (setwise),  then $\gcompMN^s$ is elementarily indivisible and has infinitely many orbits. In particular, by Lemma \ref{infChain}, there is such a structure in a finite language.
\end{theorem}

\begin{proof}
	Without loss of generality, $T_1:=\Th(\cM)$  admits QE and there is an $\cL$-theory $T_2$, admitting QE, such that  $\cN_a \models T_2$ for all $a\in \cM$. Let $T$ be as guaranteed by Theorem \ref{compositionQE}. 
	So $\scomp{\cM}{\cN}^s, \gcompMN^s \models T$ and obviously $\scomp{\cM}{\cA_0}^s$ can be embedded into $\gcompMN^s$ and vice versa. By QE of $T$, these embeddings are elementary, so $\scomp{\cM}{\cA_0}^s \sim_e \gcompMN^s$.
	By Theorem \ref{compElemInd}, $\scomp{\cM}{\cA_0}^s$ is elementarily indivisible and thus by Lemma \ref{elemEquivResult} so is $\gcompMN^s$.
	
	Now, since, by choice of $\{\cN_a\}_{a\in M}$, there are infinitely many pairwise non-isomorphic $\cN_a$s, by Lemma \ref{automorphisInduces}, $\gcompMN^s$ has infinitely many orbits.
\end{proof}

\section{An elementarily non-symmetrically indivisible structure}\label{elemNonSym}

In this section we will provide an negative answer to Question \ref{q:elemSym}.

But first, we provide a simpler construction of an indivisible structure that is not symmetrically indivisible. This construction is given to provide the reader with intuition for the continuation of this section and will be generalized in Proposition \ref{compNotSym}. The quick reader may skip the following example.

\begin{example}\label{simpleNonSym}
	Let $\cL = \{<\}$, let $\omega$ be the $\cL$-structure of ordered natural numbers and  let $X$ a pure countably infinite set (letting $<^X = \emptyset $). Then $\scomp{X}{\omega}$ is indivisible but not symmetrically indivisible.
\end{example}

\begin{proof}
	$\scomp{X}{\omega}$ is indivisible by Proposition \ref{compositionIndivisible}. As for symmetric indivisibility -- let $\{x_i\}_{i\in \omega}$ be an enumeration of $X$ and $c:\scomp{X}{\omega}\to \{\text{red},\text{blue} \}$ be the coloring defined as follows:
	\[ c\left(\left(x_i, j\right)\right) := \twopartdef{\text{red}}{j\leq i}{\text{blue}}{j>i.}  \]
	Every monochromatic red substructure will have only finite $<$-chains, and thus not isomorphic to $\scomp{X}{\omega}$. It is left to show that there is no monochromatic blue symmetrically embedded substructure isomorphic to $\scomp{X}{\omega}$. Assume towards contradiction $\cB$ is such a structure and let $\left(x_{i_0}, j_0\right)\in \cB$. Since $\cB \cong \scomp{X}{\omega}^s$, $\cB$ has infinitely many infinite $<$-chains and every chain is of the form $\cB\cap \left(\{x_i\}\times \omega\right)$. So let $i_1>j_0$ be such that $\cB \cap \left(\{x_{i_1}\}\times \omega\right) \neq \emptyset$. Let $\sigma\in \aut(\cB)$ be such that
	$ \sigma[\cB \cap \left(\{x_{i_0}\}\times \omega\right)] = \cB \cap \left(\{x_{i_1}\}\times \omega\right) $ and let $\left(x_{i_1},j_1\right):= \sigma\left(\left(x_{i_0},j_0\right)\right)$. Since $\left(x_{i_1},j_1\right)\in \cB$ and $\cB$ is all-blue, $j_1>i_1>j_0$. Since $\cB$ is symmetrically embedded, there is an automorphism $\widetilde{\sigma}\in \aut\left(\scomp{X}{\omega}^s\right)$ extending $\sigma$. Define $\tau \in \aut\left(\scomp{X}{\omega}^s\right)$ as follows:
	\[ \tau\left(\left(x_i,j\right)\right):=\threepartdef{\left(x_{i_1},j\right)}{i = i_0}  {\left(x_{i_0},j\right)}{i = i_1}  {\left(x_{i},j\right)}{i \neq i_0, i_1.} \]
	Namely, $\tau$ is the automorphism swapping $\{x_{i_0}\}\times\omega$ and $\{x_{i_1}\}\times\omega$.
	
	Now $\tau \circ \widetilde{\sigma} [\{x_{i_0}\}\times\omega] = \{x_{i_0}\}\times\omega$, so $\tau\circ \widetilde{\sigma} \upharpoonright \left(\{x_{i_0}\}\times\omega\right)$ is an automorphism of $\{x_{i_0}\}\times\omega$ and $ \tau\circ \widetilde{\sigma}\left(\left(x_{i_0},j_0\right)\right) = \left(x_{i_0},j_1\right)$. This is a non-trivial automorphism of $\{x_{i_0}\}\times \omega$, but $\left(\{x_{i_0}\}\times \omega\right) \cong \omega$ is rigid.
\end{proof}

\begin{lemma}\label{automorphism}
	Let $\cM, \cN$ be $\cL$-structures and let $\sigma \in \aut(\cM)$. If $\widetilde{\sigma}:\scomp{\cM}{\cN}^s \to \scomp{\cM}{\cN}^s$ is defined by
	$\widetilde{\sigma} \left(\left(a,b\right)\right) = \left(\sigma(a),b\right)$, then $\widetilde{\sigma}$ is an automorphism.
	
	In particular, if $\cM$ is transitive and $A, B\subset \scomp{\cM}{\cN}^s$ are $s$-equivalence classes, then there is an automorphism $\tau \in \aut\left(\scomp{\cM}{\cN}^s\right)$ such that $\tau[A] = B$.
\end{lemma}
\begin{proof}
	Clearly $\widetilde{\sigma}$ is a bijection. Notice that $\widetilde{\sigma^{-1}} = \left(\widetilde{\sigma}\right)^{-1}$, and since $\sigma$ is arbitrary, proving that $\widetilde{\sigma}$ is a homomorphism will suffice. It is clear that $\widetilde{\sigma}$ preserves $s$
	
	Let $R \in \cL$ be an $n$-ary relation, $\widebar{(a,b)}:=\left(\left(a_1,b_1\right),\dots,\left(a_n,b_n\right)\right)\in \scomp{\cM}{\cN}^s$ and assume
	$\scomp{\cM}{\cN}^s\models R\left(\widebar{(a,b)}\right)$.
	From the definition of $\scomp{\cM}{\cN}^s$, one of the following holds:
	\begin{itemize}
		\item $\bigvee_{1\leq j,k\leq n} {a_j\neq a_k}$ and $\cM\models R\left(a_1,\dots, a_n\right)$, so since $\sigma$ is an automorphism,
		$\bigvee_{1\leq j,k\leq n} {\sigma\left(a_j\right)\neq \sigma\left(a_k\right)}$ and $\cM\models R\left(\sigma\left(a_1\right),\dots, \sigma\left(a_n\right)\right)$.
		
		\item $\bigwedge_{1\leq j,k\leq n} {a_j = a_k}$ and $\cN\models R\left(b_1,\dots, b_n\right)$, so 
		$\bigwedge_{1\leq j,k\leq n} {\sigma\left(a_j\right) =  \sigma\left(a_k\right)}$ and $\cN\models R\left(b_1,\dots, b_n\right)$.
	\end{itemize}
	In any case, \[\scomp{\cM}{\cN}^s\models R\left(\widetilde{\sigma}\left(\left(a_1,b_1\right)\right),\dots, \widetilde{\sigma}\left(\left(a_1,b_1\right)\right)\right) .\]
\end{proof}
\begin{proposition}\label{compNotSym}
	If $\cM$ is a transitive structure and $\cN$ is a structure with infinitely many orbits such that $\cN$ can not be embedded into any finite union of orbits, then $\scomp{\cM}{\cN}^s$ is not symmetrically indivisible.
\end{proposition}

\begin{proof}
	We generalize the proof of Example \ref{simpleNonSym}:
	let $\{a_i\}_{i\in \omega}$ be an enumeration of $M$ and  $\{O_i\}_{i\in \omega}$ an enumeration of the orbits of $\cN$. For $b\in \cN$, denote $on(b) = j$ if $b\in O_j$ and define a coloring $c:\scomp{\cM}{\cN}^s \to \{\text{red},\text{blue} \}$ as follows:
	\[ c\left(a_i, b\right) := \twopartdef{\text{red}}{on(b)\leq i}{\text{blue}}{on(b)>i.}  \]
	For every all-red substructure, every $s$-equivalence class will be embedded in a finite union of orbits, and thus not isomorphic $\cN$. It is left to show that there is no all-blue symmetrically embedded substructure isomorphic to $\scomp{\cM}{\cN}$. Assume towards contradiction $\cB$ is such a structure and let $\left(a_{i_0}, b\right)\in \cB$. Denote $j_0:=on(b)$. Since $\cB \cong \scomp{\cM}{\cN}^s$, $\cB$ has infinitely many infinite $s$-equivalence classes and every $s$-equivalence class of $\cB$ is of the form $\cB \cap \left(\{a\}\times \cN\right)$ for some $a\in M$. Let $i_1>j_0$ such that $\cB \cap \left(\{a_{i_1}\}\times \cN\right) \neq \emptyset$. Since $\cM$ is transitive, by Lemma \ref{automorphism}, for every two $s$-equivalence classes $A, B\subset \scomp{\cM}{\cN}^s$, there is an automorphism $\tau \in \aut\left(\scomp{\cM}{\cN}^s\right)$ such that $\tau[A] = B$. Since $\cB\cong \scomp{\cM}{\cN}^s$, this is true for $\cB$ as well, so let $\tau\in \aut(\cB)$ be an automorphism such that 
	\[\tau[\cB \cap \left(\{a_{i_0}\}\times \cN\right)] = \cB \cap \left(\{a_{i_1}\}\times \cN\right) .\]
	Denote  $\left(a_{i_1},c\right):=\tau\left(\left(a_{i_0},b\right)\right)$. Since $\left(a_{i_1},c\right)$ is blue, $on(c)>i_1>j_0 = on(b)$.
	
	Since $\cB$ is symmetrically embedded, let $\widehat{\tau}\in \aut\left(\scomp{\cM}{\cN}^s\right)$ extending $\tau$. Let $\sigma\in\aut(\cM)$ such that $\sigma\left(a_{i_1}\right) = a_{i_0}$ and let $\widetilde{\sigma}\in \aut\left(\scomp{\cM}{\cN}^s\right)$ as defined in Lemma \ref{automorphism}. $\widetilde{\sigma}\circ\widehat{\tau}$ is an automorphism and $\widetilde{\sigma}\circ\widehat{\tau}[\{a_{i_0}\}\times \cN] = \{a_{i_0}\}\times \cN$, so \[ \theta:= \widetilde{\sigma}\circ\widehat{\tau} \upharpoonright \{a_{i_0}\}\times \cN \] is an automorphism of $\{a_{i_0}\}\times \cN$. Define $\iota_0:\cN \overset{\cong}{\rightarrow} \{a_{i_0}\}\times \cN$ by $\iota_0(b):=\left(a_{i_0},b\right)$. $\iota_0^{-1}\circ \theta \circ \iota$ is an automorphism of $\cN$ and $\iota_0^{-1}\circ \theta \circ \iota (b) = c$, but this contradicts $on(c)>on(b)$.
\end{proof}

\begin{theorem}\label{elemNonSymResult}
	There is an elementarily indivisible structure (in a finite language) that is not symmetrically indivisible.
\end{theorem}

\begin{proof}
	Let $\cA:=\gcompMN^s$ as in Theorem \ref{infOrb} and let $\cB$ be any elementarily indivisible transitive structure (in a finite language). If we choose $\{\cN_a\}_{a\in M}$ such that $\Set{a\in M | \cA_i = \cN_a}$ is finite for every $i\in \omega$, then by Lemma \ref{automorphisInduces} every orbit of $\cA$ has only finitely many $s$-equivalence classes and $\cA$ can not be embedded into only finitely many orbits.  By Theorem \ref{compElemInd}, $\scomp{\cB}{\cA}^s$ is elementarily indivisible, but by Proposition \ref{compNotSym}, it is not symmetrically indivisible.
\end{proof}

\subsection*{Acknowledgements}
	The work in this paper is part of the author's M.Sc. thesis, prepared under the supervision of Assaf Hasson. The author would like to gratefully acknowledge him for presenting the questions discussed in the paper, as well as for fruitful discussions and the great help and support along the way. The author was partially supported by an Israel Science Foundation grant number 1156/10.

\bibliographystyle{alpha}
\bibliography{thesisref}

\end{document}